\documentclass[leqno,11pt]{article}
\usepackage{amssymb,amsmath,amsfonts,amsthm}
\usepackage{epsfig,color,mathrsfs}
\usepackage{graphicx}
\usepackage{amscd}
\usepackage{caption}
\usepackage{mathrsfs}
\usepackage{multicol}
\usepackage{color}

\newcommand{\be}{\begin{equation}}
\newcommand{\ee}{\end{equation}}
\newcommand{\ben}{\begin{eqnarray*}}
\newcommand{\een}{\end{eqnarray*}}

\allowdisplaybreaks

\newtheorem{theorem}{Theorem}[section]

\newtheorem{corollary}[theorem]{Corollary}

\newtheorem{lemma}[theorem]{Lemma}

\definecolor{darkgreen}{rgb}{0.09, 0.45, 0.27}
\definecolor{debianred}{rgb}{0.84, 0.04, 0.33}

\begin{document}
\title{\vskip-0.3in Fujita type results for quasilinear parabolic inequalities with nonlocal terms}

\author{Roberta Filippucci\footnote{Dipartimento di Matematica e Informatica, Universit\'a degli Studi di Perugia, Via Vanvitelli 1, 06123 Perugia, Italy; {\tt roberta.filippucci@unipg.it}}
      $\quad$   and    $\quad$
{Marius Ghergu\footnote{School of Mathematics and Statistics,
    University College Dublin, Belfield, Dublin 4, Ireland; {\tt
      marius.ghergu@ucd.ie}}\;\,\footnote{Institute of Mathematics Simion Stoilow of the Romanian Academy, 21 Calea Grivitei St., 010702 Bucharest, Romania}}
}


\maketitle

\begin{abstract} In this paper we investigate  the nonexistence of nonnegative solutions  of parabolic inequalities of the form
$$\begin{cases}
&u_t \pm L_\mathcal A u\geq  (K\ast u^p)u^q \quad\mbox{ in } \mathbb R^N \times \mathbb (0,\infty),\, N\geq 1,\\
&u(x,0) = u_0(x)\ge0 \,\, \text{ in } \mathbb R^N,\end{cases}
\qquad (P^{\pm})
$$
where  $u_0\in L^1_{loc}({\mathbb R}^N)$, $L_{\mathcal{A}}$  denotes a  weakly
 $m$-coercive operator, which includes  as prototype  the $m$-Laplacian or the generalized mean curvature operator, $p,\,q>0$, while $K\ast u^p$ stands for  the standard convolution operator between a weight $K>0$ satisfying suitable conditions at infinity and $u^p$. 
For problem $(P^-)$ we obtain a Fujita type exponent while for $(P^+)$ we show that no such critical exponent exists.  Our approach relies on  nonlinear capacity estimates adapted to the nonlocal setting of our problems. No comparison results or maximum principles are required.
\end{abstract}

\noindent{\bf Keywords:} Quasilinear parabolic inequalities; nonlocal terms; Fujita exponent.

\medskip

\noindent{\bf 2010 AMS MSC:} 35K59,  35A23, 35B33, 35B53 


\section{Introduction and the main results}\label{sec1}
In this paper we deal with  the nonexistence of nonnegative solutions of the following parabolic  problems
\begin{equation}\label{main}
\begin{cases}
&u_t-L_\mathcal A u\geq  (K\ast u^p)u^q \quad\mbox{ in } \mathbb R^N \times \mathbb (0,\infty),\, N\geq 1\\
&u(x,0) = u_0(x)\ge0 \,\, \text{ in } \mathbb R^N, 
\end{cases}
\end{equation}
and
\begin{equation}\label{main2}
\begin{cases}
&u_t+L_\mathcal A u\geq  (K\ast u^p)u^q \quad\mbox{ in } \mathbb R^N \times \mathbb (0,\infty),\\
&u(x,0) = u_0(x)\ge0 \,\, \text{ in } \mathbb R^N, N\geq 1,\end{cases}
\end{equation}
where $u_0\in L^1_{loc}({\mathbb R}^N)$, $u_0\geq 0$  and the differential operator $L_{\mathcal{A}}u= \text{div} \mathcal{A}(x,u,\nabla u)$ is such that the mapping  $\mathcal{A}: \mathbb R^N \times \mathbb{R}\times \mathbb{R}^N \to \mathbb{R}^N$ is a Carath\'{e}odory function with
\begin{equation}\label{A_cal}\mathcal{A}(x,z,0) = 0,\qquad  \mathcal{A}(x,z,\xi)\cdot \xi\geq 0\end{equation} 
for every $(x,z,\xi) \in \mathbb{R}^N\times [0, \infty)\times\mathbb{R}^N$. In addition, $L_{\mathcal{A}}$  is assumed to be  a weakly $m$-coercive operator, that is,  there exist  a  constant $c_0> 0$ and an exponent $m>1$
such that the inequality \begin{equation} \label{weak_m}
 \mathcal A(x,z,\xi)\cdot \xi\ge c_0 |\mathcal A(x,z,\xi)|^{m'}
\end{equation}
holds for all $(x,z,\xi) \in \mathbb{R}^N\times [0, \infty)\times\mathbb{R}^N$ with $m'= m /(m-1)$.

The main prototypes for $\mathcal{A}$ are:
\begin{itemize}
\item the $m$-Laplace operator given by $\mathcal{A}(x,z,\xi)=|\xi|^{m-2}\xi$, $m>1$;
\item the $m$-mean curvature operator, or the generalized mean curvature operator, given by
$$
\mathcal{A}(x, z, \xi)= \frac{|\xi|^{m-2}}{\sqrt{1+|\xi|^m}}\xi.
$$
\end{itemize}
Furthermore, on the right-hand side  of \eqref{main} and \eqref{main2} we assume the exponents $p,q>0$. The function  $K\in C(\mathbb{R}^N\setminus\{0\})$, $K>0$ satisfies $\liminf_{x\to 0}K(x)>0$  and there exists $\rho>0$ and $0<\beta<m/2$ such that
\begin{equation}\label{la}
K(x)\geq c|x|^{-\beta} \quad \mbox{ for all }\; x\in \mathbb{R}^N, \, |x|>\rho,
\end{equation}
where $c>0$ is a positive constant. Also, by $K\ast u^p$ we denote the standard convolution operator defined by
$$
(K\ast u^p)(x,t)=\int_{\mathbb R^N}K(x-y)u^p(y,t) dy\quad\mbox{ for all }(x,t)\in\mathbb R^N\times (0, \infty).
$$
We are interested in  {\it nonnegative weak solutions} of \eqref{main},  that is nonnegative functions $u(x,t)$, belonging to the class $\mathcal S$ given by those  $u\in  {W}^{1,m}_{\rm loc}(\mathbb R^N\times (0, \infty))$ which fulfil the two conditions below
\begin{enumerate}
\item[(i)] ${\mathcal{A}(x,u,\nabla u)  \in  [{L}^{m'}_{\rm loc}(\mathbb R^N\times (0, \infty))]^N}$,
\item[(ii)] $(K\ast u^p)u^q  \in  {L}^1_{\rm loc}(\mathbb R^N\times (0, \infty))$,
\end{enumerate}
and such that  for any nonnegative test function  $\varphi \in C^{1}_{c}(\mathbb R^N\times \mathbb{R})$, we have
\begin{equation} \label{weak_form2}
\int_0^\infty\!\int_{\mathbb R^N} (K\ast u^p)u^q  \varphi \, dx \, dt \leq  \int_0^\infty\!\int_{\mathbb R^N} u_t \; \varphi \, dx \, dt
 + \int_0^\infty\!\int_{\mathbb R^N} \mathcal{A}(x,u,\nabla u) \cdot \nabla \varphi \, dx \, dt,
\end{equation}
or equivalently
\begin{equation} \label{weak_form1}
\begin{aligned}
\int_0^\infty\!\!\int_{\mathbb R^N} (K\ast u^p)u^q \varphi \, dx \, dt \leq & - \int_{{\mathbb{R}^N}} \, u_0(x) \varphi(x,0) \, dx - \int_0^\infty\!\!\int_{\mathbb R^N} u \, \varphi_{t} \, dx \, dt \\
& + \int_0^\infty\!\!\int_{\mathbb R^N} \mathcal{A}(x,u,\nabla u) \nabla\varphi \, dx \, dt. \end{aligned}
\end{equation}

For {\it nonnegative weak solutions} of \eqref{main2} the functions $u$ in class $\mathcal{S}$ satisfy (i)-(ii) above and we replace \eqref{weak_form2} by 
\begin{equation} \label{weak_form3}
\int_0^\infty\!\int_{\mathbb R^N} (K\ast u^p)u^q  \varphi \, dx \, dt \leq  \int_0^\infty\!\int_{\mathbb R^N} u_t \; \varphi \, dx \, dt
 - \int_0^\infty\!\int_{\mathbb R^N} \mathcal{A}(x,u,\nabla u) \cdot \nabla \varphi \, dx \, dt.
\end{equation}

The pioneering paper in this subject is  due to Fujita \cite{Fuj} in $1966$, where he investigated the fol\-lo\-wing Cauchy problem for the semilinear heat equation
\begin{equation}\label{fujita66}
\begin{cases}
&u_t - \Delta u = u^{q} \,\,\, \text{in }\mathbb{R}^{N} \times (0, \infty),\\
&u(x,0) = u_0(x) \,\, \text{in } \mathbb{R}^{N}.\\
\end{cases}
\end{equation}
Fujita obtained the critical exponent $ q_{F} = 1 + 2/N$ on the existence versus nonexistence of nonnegative nontrivial solutions. Precisely,  nonexistence of solutions, that is blow up, holds when $1<q< 1 + 2/N$ and $u_0$ is bounded and nonnegative, while blow up can occur when
 $q> 1 + 2/N$ depending on the size of $u_0$. Since then, there have been a number of extensions of Fujita results in many directions.
In particular, the result obtained by Fujita was completed, relatively to the critical case, in \cite{Hay} for $N=1,2$ and in \cite{KST} for $N\geq 3$. 

When one moves from the equality to the inequality case, in contrast to elliptic problems, the critical exponents for parabolic equations and inequalities coincide,  at least in the semilinear case.  This observation  is due to Mitidieri and Pokhozev in \cite[Part II]{MPbook} where a powerful tool on a priori estimates is devised. Specifically, 
the technique of the nonlinear capacity developed by the authors in \cite{MPbook} is based on the derivation of suitable a priori bounds on the solution of the problem under consideration by carefully choosing special test functions and scaling argument (other types of a priori estimates for quasilinear parabolic problems may be found in \cite{LS2009, LS2013, LSS2011}). This fact enables one to analyze by the same method nonexistence of solutions, not necessarily nonnegative, of essentially different problems, namely noncoercive parabolic problems 
$$u_t-\Delta u\ge |u|^q \quad\text{in }\mathbb{R}^{N} \times (0, \infty),$$
and coercive parabolic problems respectively given by
\begin{equation}\label{coer_mod}u_t+\Delta u\ge |u|^q \quad\text{in }\mathbb{R}^{N} \times (0, \infty).
\end{equation}
As noted in \cite[Remark 3.8]{MP_JEv}, it is possible to show that in this case the problem \eqref{coer_mod} has no non negative weak solutions for every $q > 1$.

We also refer the reader to \cite[Section 18]{QS} where the authors propose a modification of the above arguments, based on rescaling of test functions.

Later on, several extensions to quasilinear problems were developed in literature: see e.g., the book by
Samarskii, Galaktionov, Kurdyumov and Mikhailov \cite{SGKM} and the work by Levine, Lieberman and Meier \cite{LLM} on parabolic problems involving the mean curvature operator.

For Fujita type results for degenerate parabolic inequalities and for systems of quasilinear parabolic inequalities  we refer  the reader to \cite{MP_JEv} and to \cite{MP}. In particular, the Fujita exponent for the $m$-Laplacian parabolic case
\begin{equation}\label{fujita661}
\begin{cases}
&u_t - \Delta _mu \ge u^{q} \,\,\, \text{in }\mathbb{R}^{N} \times (0, \infty),\\
&u(x,0) = u_0(x)\ge0 \,\, \text{in } \mathbb{R}^{N},\\
\end{cases}
\end{equation}
is obtained in  \cite{MP}  where it is proved that \eqref{fujita661} has non nonnegative nontrivial solutions for
$$\max\{1,m-1\}<q\le m-1+\frac mN.$$ This last condition forces $m>2N/(N+1)$. 
For further generalizations we refer the reader to 
\cite{zhou} for inhomogeneous pseudoparabolic equations, to \cite{KK} for  parabolic problems in halfspace, or to \cite{FL} for parabolic inequalities with weights and  nonlinearites depending on the gradient.

Nonlocal models describe many natural phenomena, such as the non-Newton flux in the mechanics of fluid, population of biological species and filtration.
Concerning a nonlocal sources problem, Galaktionov and Levine \cite{GL} investigated positive solutions of a Cauchy problem for the following semilinear parabolic equation with weighted nonlocal sources
\begin{equation}\label{Gal_Lev} 
\begin{cases}
&u_t=\Delta u^m +\biggl( \int_{\mathbb R^N}K(x)u^p(x,t)dx\biggr)^{(r-1)/p}u^q\,\,\, \text{in }\mathbb{R}^{N} \times (0, \infty),\\
&u(x,0) = u_0(x)\ge0 \,\, \text{in } \mathbb{R}^{N},\\
\end{cases}
\end{equation}
where $p, q, r\ge1$, $m>1$ and $K$ is a  function not necessarily in $L^1(\mathbb R^N)$. Other semilinear parabolic problems with nonlocal terms are discussed in \cite[Part V]{QS}.

This paper is a first attempt in solving quasilinear parabolic problems of type \eqref{main} and \eqref{main2} which feature  nonlocal terms defined by the convolution operation. Such convolution terms are used to model various quantities in gravity and quantum physics. For instance, the equation
\begin{equation}\label{hartree}
i\psi_t-\Delta \psi=(|x|^{\alpha-N}\ast \psi^2)\psi\quad\mbox{in }{\mathbb R}^N, \alpha\in (0, N), N\geq 1,  
\end{equation}
was introduced in 1928 by Hartree \cite{H1,H2,H3}, shortly after the publication of the Schr\"odinger equation, in order to study the  non-relativistic atoms, using the concept of self-consistency. The stationary case of \eqref{hartree} for $N = 3$ and $\alpha = 2$ is known in the literature
as the {\it Choquard} (or {\it Choquard-Pekar}) equation and was introduced in \cite{P} as a model in quantum theory (see also \cite{MS} for a mathematical account on this equation). Stationary singular solutions of the Choquard inequality are discussed in \cite{GT} while quasilinear elliptic equations and inequalities with convolution terms are considered in \cite{FG, GKS2020, GKS2021, GMM2021}.

Our main result concerning \eqref{main} is the following.

\begin{theorem}\label{thmain}
Assume
\begin{equation}\label{cond_beta_m} 0<\beta<m/2\qquad\mbox{ and } \qquad m >\frac{2N+1}{N+1}\bigl(\in (1,2)\bigr).\end{equation}
If
\begin{equation}\label{cond1}
2\max\{1,m-1\}<p+q\le  m-1+\frac{N-\beta+m}{N+\beta} ,
\end{equation}
then problem \eqref{main} does not have nonnegative nontrivial solutions  $u\in \mathcal{S}$.
\end{theorem}

In particular, condition \eqref{cond1} yields the following upper bounds for $\beta$ 
\begin{equation}\label{cond2_beta}
\beta<
\begin{cases}
\displaystyle 1-\frac{m-2}{m}N &\quad\mbox{ if } 2<m<\frac{2N}{N-1},\\[0.1in]
\displaystyle \frac{m-(2-m)N}{4-m} &\quad\mbox{ if }\frac{2N+1}{N+1}<m<2.
\end{cases}
\end{equation}

The proof of Theorem \ref{thmain} is carried out through nonlinear capacity estimates specifically adapted to the nonlocal setting of our problem and the weak $m$-coercivity of the differential operator $L_{\mathcal{A}}$. Precisely, we derive integral estimates in time for the new quantity
$$
J(t)=\int_{\mathbb R^N}u^\ell (x,t)\psi^k(x, t)dx\, , \,\,\, t\geq 0,
$$
where $\ell=(p+q)/2$ and $\psi$ is a carefully selected test function  with compact support (see \eqref{testfunction}). As corollaries of our main results we obtain:

\begin{corollary} Assume
$$ 
0<\beta<1 \quad\text{and}\quad 
2<p+q\le 2+\frac{2(1-\beta)}{N+\beta}.
$$
Then problem 
$$\begin{cases}
&\displaystyle u_t-\mbox{\rm div} \biggl(\frac{\nabla u}{\sqrt{1+|\nabla u|^2}}\biggr)\geq  (K\ast u^p)u^q \quad\mbox{ in } \mathbb R^N \times \mathbb (0,\infty),\\[0.2in]
&u(x,0) = u_0(x)\ge0 \,\, \text{ in } \mathbb R^N,\end{cases}
$$ 
does not have nonnegative nontrivial solutions.
\end{corollary}

\begin{corollary}
Assume
$$0<\beta<m/2 \quad \mbox{ and }\quad m >\frac{2N+1}{N+1}\bigl(\in (1,2)\bigr).$$
If  \eqref{cond1} holds,
then the problem
$$\begin{cases}
&u_t-\Delta_m u\geq  (K\ast u^p)u^q \quad\mbox{ in } \mathbb R^N \times \mathbb (0,\infty),\\
&u(x,0) = u_0(x)\ge0 \,\, \text{ in } \mathbb R^N,\end{cases}
$$ 
does not have nonnegative nontrivial solutions.
\end{corollary}

The critical Fujita exponent also depends on the asymptotic behavior at infinity of the initial function $u_0(x)$. The result below is an extension of Lemma 26.2 in \cite{MP} to the setting of nonlocal problems.

\begin{corollary}\label{cornu} Assume that $u_0\in L^1_{loc}(\mathbb{R}^N)$ and for all $R>0$ large enough $u_0$ satisfies
\begin{equation}\label{u_0_nu}
\int_{B_R}u_0(x)dx\ge c R^\nu\,,
\end{equation}
for some positive constant $c$ and an exponent $0\le \nu< N+\beta$, $0<\beta<(m+\nu)/2$.
If
 \begin{equation}\label{critic_nu}
2\max\{1,m-1\}<p+q<m-1+\frac{N-\beta+m}{N+\beta-\nu},
\end{equation}
then problem  \eqref{main} does not have nonnegative nontrivial solutions $u\in \mathcal{S}$.
\end{corollary}
In this case, condition \eqref{critic_nu} yields the following upper bound for $\beta$ 
\begin{equation}\label{cond2_beta_nu}
\beta<
\begin{cases}
\displaystyle 1-\frac{(m-2)N-(m-1)\nu}{m} &\quad\mbox{ if } 2<m<\frac{2N}{N-1},\\[0.1in]
\displaystyle \frac{m-(2-m)N+(3-m)\nu}{4-m} &\quad\mbox{ if }\frac{2N+1}{N+1}<m<2.
\end{cases}
\end{equation}
In particular, when $m>2$, if we restrict the range of $\nu$ to the set
$$\frac{N(m-2)-m}{m-1}<\nu<N+\beta,$$
then the upper bound $2N/(N-1)$  on  $m$ in \eqref{cond2_beta_nu} can be removed.

Next, we turn to the study of \eqref{main2}. In this setting we derive our a priori estimates employing similar arguments to those we used in the proof of Theorem \ref{thmain}. Unlike the approach for \eqref{main} where the  a priori estimates are obtained by choosing two classes of test functions in \eqref{weak_form2}, for the counterpart problem \eqref{main2} we can only use a single class of test functions (this was already emphasised in \cite[Section 26]{MPbook} for local coercive parabolic problems). We shall overcome this fact by imposing a higher  locally integrability on the solution. The precise solution space will be defined in what follows. 

Assume 
\begin{equation}\label{new_beta_coerc}0<\beta<m/2\quad\text{and}\quad p+q>2\max\{1,m-1\}.
\end{equation} 
Let $d>0$ be such that
\begin{equation}\label{eqd}
2<\Gamma(d):=\max\Big\{\frac{p+q+d}{d+1}\, , \, \frac{p+q+d}{d+m-1}  \Big\}<\frac{2N+m}{N+\beta}.
\end{equation}
Such a value $d>0$ always exists since $\Gamma$ is decreasing as a function of $d$ and 
$$
\Gamma(0)>2 \quad\mbox{ and }\quad \lim_{d\to \infty}\Gamma(d)=1<\frac{2N+m}{N+\beta}.
$$

Our main result concerning the problem \eqref{main2} is the following.

\begin{theorem}\label{thmain2}
If \eqref{new_beta_coerc} and \eqref{eqd} hold, 
then problem \eqref{main2} does not have nonnegative nontrivial solutions 
 $u$ in the class
$$
\mathcal{S}\cap \Big\{(K\ast u^p)u^{q+d} \in  L^{1}_{loc}(\mathbb R^N\times [0, \infty)) \Big\}.
$$
\end{theorem}
In particular, if \eqref{new_beta_coerc} holds, then \eqref{main2} has no nonnegative nontrivial solutions $u\in \mathcal{S}\cap C(\mathbb R^N\times (0, \infty))$.

\section{Proof of Theorem \ref{thmain}}

Suppose by contradiction that \eqref{main} admits a nonnegative nontrivial solution 
$u\in \mathcal{S}$. We start with the following result which provides extra local integrability of $u$.

\begin{lemma}\label{lp0}
Let $u\in \mathcal{S}$ be a nonnegative solution of \eqref{main}. Then,
\begin{equation}\label{uLp+q_2}
u^{(p+q)/2}\in L^1_{loc}(\mathbb{R}^N\times [0,\infty)).
\end{equation}
\end{lemma}
\begin{proof} Let $R>\rho$ be large, where $\rho>0$ appears in \eqref{la}. For $x\in B_R(0)$, using \eqref{la} we have
\begin{equation}\label{KKK1}
\begin{aligned}
(K\ast u^p)(x,t)& \geq \int_{\mathbb B_{R}} K(x-y)u^p(y,t) dy\\
&=\int_{|x-y|\leq \rho} K(x-y)u^p(y,t))dy+\int_{\substack{|x-y|>\rho\\  |x|,\, |y|<R}} K(x-y)u^p(y,t)dy\\
&\geq \inf_{z\in B_1(0)} K(z) \int_{|x-y|\leq \rho} u^p(y,t)dy+c\int_{\substack{|x-y|>\rho \\  |x|,\, |y|<R}}  |x-y|^{-\beta} u^p(y,t) dy\\
&\geq \inf_{z\in B_1(0)} K(z)  \int_{|x-y|\leq \rho} u^p(y,t) dy+c(2R)^{-\beta}  
\int_{|x-y|>\rho ,\, |y|<R}  u^p(y,t) dy\\
&\geq CR^{-\beta}\left\{ \int_{|x-y|\leq \rho\, , \,  |y|<R} u^p(y,t)dy+\int_{|x-y|>\rho ,\, |y|<R}  u^p(y,t) dy \right\} \\
&\geq CR^{-\beta}\int_{B_R(0) } u^p(y,t) dy,
\end{aligned}
\end{equation}
provided $R>\rho$ is large enough.

Next, using \eqref{KKK1} and H\"older's inequality we find
$$
\begin{aligned}
\infty&>\int_0^T\!\!\int_{B_{R}(0)} \big(K\ast u\big)(x,t)u^q(x,t) dxdt \\
&\geq CR^{-\beta}\int_0^T\Big(\int_{B_R(0)} u^p(x) dx\Big)   \Big(\int_{B_R(0)} u^q(x) dx\Big) dt \qquad \mbox{  (by \eqref{KKK1})}\\
&\geq CR^{-\beta} \int_0^T \Big(\int_{B_R(0)} u^{(p+q)/2}(x) dx\Big)^2 dt \qquad \mbox{  (by H\"older's inequality on $B_R(0)$)}\\
&\geq \frac{CR^{-\beta}}{T} \Big( \int_0^T \int_{B_R(0)} u^{(p+q)/2}(x) dx \, dt \Big)^2 \qquad \mbox{  (by H\"older's inequality on $[0,T]$)}\\
\end{aligned}
$$
which shows that $u^{(p+q)/2}\in L^1(B_R(0)\times [0,T])$ and concludes our proof. 
\end{proof}

From \eqref{cond1} we have $(p+q)/2>\max\{1, m-1\}$. We may thus choose $d>0$ small enough such that $(p+q)/2>(m-1)(d+1)$. Using  H\"older's inequality we deduce
\begin{equation}\label{Summ_with_d}
u^{1-d}\, , \, u^{(m-1)(1-d)}\,,\, u^{(m-1)(d+1)}\in L^1_{loc}(\mathbb{R}^N\times [0,\infty)).
\end{equation}
This will ensure that all integrals in this section are finite.

The proof of Theorem \ref{thmain} will be achieved along a several lemmas. First we want to precise the choice of the test functions in \eqref{weak_form2}. Take a standard cut off function $\xi \in C^{1}[0,\infty)$ such that
\begin{itemize}
\item $\xi = 1$ in $(0,1)$, $\xi= 0$  in $(2,\infty)$;
\item  $0\le \xi\le 1$ and $|\xi'| \leq C$ in $[0, \infty)$, for some $C>0$.
\end{itemize}

Now take $R > 0$ and consider the functions
\begin{equation} \label{cutoff}
\chi_R (x) = \xi \biggl(\frac{|x|}{R}\biggr) ,\quad \eta_{R}(t ) = \xi\biggl(\frac{t}{R^{\gamma}} \biggr),
\end{equation}
with $\gamma\ge1$. Clearly $\chi_R (x)=1$ in $B_R(0)$, where $B_R(0)$ denotes  the open ball in $\mathbb{R}^N$, centered at the origin and having radius $R > 0$.

Finally define, for all $R > 0$, the nonnegative cut off function $\psi:\mathbb R^N\times [0, \infty)\to [0, \infty)$, given by
\begin{equation} \label{testfunction}
\psi(x,t) = \chi_R(x) \, \eta_{R}(t).
\end{equation}

Note that,  $\psi$ is the restriction to  $\mathbb R^N\times [0, \infty)$ of a 
$C^1_c(\mathbb R^N\times \mathbb R)$ function, and,  as in \cite[Lemma 3.1]{FL}, by the shape of $\psi$ in \eqref {testfunction} the following inequalities hold for $\varsigma >1$, $\gamma\ge1$, $k$ and $R$ sufficiently large, say $k>\varsigma$,
\be\label{PSI}
\int_0^\infty\!\int_{\mathbb R^N}{\psi^{k-\varsigma}}{|\nabla\psi|^{\varsigma}}dxdt\le c R^{-\varsigma +N +\gamma},\ee
\be\label{PSII} \int_0^\infty\!\int_{\mathbb R^N}{\psi^{k-\varsigma}}{|\psi_t|^{\varsigma}} dxdt \le c R^{-\gamma\varsigma +N +\gamma},
\ee
where $c>0$ is a constant.
\begin{lemma}\label{lp1}
Let $u\in \mathcal{S}$ be a nonnegative solution of \eqref{main} and let $\psi$ be defined by \eqref{testfunction}. Then,
\begin{equation}\label{ineq2.0}
\begin{aligned} 
\int_0^\infty\int_{\mathbb{R}^N} & (K\ast u^p)u^{q-d} \psi^{k} \, dx \, dt+
\int_0^\infty\!\!\int_{\mathbb R^N}  u^{- d - 1} \psi^{k} |\mathcal{A}(x,u,\nabla u)|^{m'}\, dx \,  dt  \\
&\leq   c_1\int_0^\infty\!\!\int_{\mathbb R^N}  u^{1 - d}\psi^{k - 1}|\psi_t| \, dx  dt + 
c_2 \int_0^\infty\!\!\int_{\mathbb R^N}  u^{ m- d- 1}\psi^{k - m} |\nabla \psi|^m  dx \, dt,
\end{aligned}\end{equation}
for some constants $c_1,c_2>0$ and $d\in(0,1)$. 
\end{lemma}

\begin{proof} Let $\epsilon > 0$ be sufficiently small and let $\{\xi_{\epsilon}\}_{\epsilon>0}$ be a standard family of mollifiers. For $\tau > 0$ we define
$$
\tilde{u}_\epsilon(x,t)  :=  \tau + \int_{\mathbb R^N} \, \xi_{\varepsilon}(x-y,t) u(y,t) \, dy \quad \text{and} \qquad u_{\tau}(x,t) := \tau + u(x,t)
$$
for $(x,t)\in {\mathbb R}^N\times (0, \infty)$. Clearly $\tilde{u}_{\epsilon}, u_{\tau} \geq \tau > 0$. In particular, since
$u \in L^{1}_{\rm loc} (\mathbb R^N\times (0, \infty))$, we have $\tilde{u}_{\epsilon} \in C^1(\mathbb R^N\times (0, \infty))$ so that the function
$
\varphi(x,t) = \tilde{u}_{\epsilon}^{-d} \; \psi^{k}(x,t)\geq 0,
$
with $d\in(0,1)$ sufficiently small so that \eqref{Summ_with_d} holds, $k$ positive to be chosen and  $\psi(x,t)$ defined in $\eqref{testfunction}$,
can be used as test function in the weak formulation of $\eqref{main}$, given by $\eqref{weak_form2}$, so that
$$\begin{aligned}
\int_0^\infty\int_{\mathbb{R}^N} (K\ast u^p)&u^q\tilde{u}_\epsilon^{-d} \psi^{k} \, dx \, dt +d \int_0^\infty\int_{\mathbb{R}^N} \psi^{k}\tilde{u}_{\epsilon}^{-d-1}\,\mathcal{A}(x,u,\nabla u) \cdot \nabla\tilde{u}_{\epsilon} \, dx \, dt \\& \le \int_0^\infty\int_{\mathbb{R}^N}  u_t \;
\tilde{u}_{\epsilon}^{-d} \; \psi^{k} \, dx \, dt
 + k  \int_0^\infty\int_{\mathbb{R}^N} \psi^{k-1} \tilde{u}_{\epsilon}^{-d} \mathcal{A}(x,u,\nabla u) \cdot  \nabla\psi \, dx \, dt. \\
\end{aligned}$$

Since $\tilde{u}_{\epsilon} \to u_{\tau}$ in $W^1_{\rm loc}(\mathbb{R}^N)$ as $\varepsilon \to 0$,
using Lebesgue dominated convergence theorem, and being $\nabla u_\tau=\nabla u$, we arrive at
\begin{equation}\begin{aligned}
\int_0^\infty\int_{\mathbb{R}^N} &(K\ast u^p)u^q u_\tau^{-d} \psi^{k} \, dx \, dt +d\int_0^\infty\int_{\mathbb{R}^N} \psi^{k} u_\tau^{-d-1}\,\mathcal{A}(x,u,\nabla u) \cdot \nabla u \, dx \, dt
\\& \le \int_0^\infty\int_{\mathbb{R}^N}  u_t \,
u_\tau^{-d} \psi^{k} \, dx \, dt
 + k  \int_0^\infty\int_{\mathbb{R}^N}  \psi^{k-1}u_\tau^{-d} |\mathcal{A}(x,u,\nabla u)| |\nabla\psi| \, dx \, dt.
\end{aligned}\end{equation}
Using the weak $m$-coerciveness of $\mathcal{A}$, from \eqref{weak_m} we deduce
\begin{equation}\label{ineq0}\begin{aligned}
\int_0^\infty\int_{\mathbb{R}^N} &(K\ast u^p)u^q u_\tau^{-d} \psi^{k} \, dx \, dt +
dc_0  \int_0^\infty\!\!\int_{\mathbb{R}^N} \psi^{k} u_\tau^{-d-1}\,|\mathcal{A}(x,u,\nabla u)|^{m'} \, dx \, dt
\\ & \le \int_0^\infty \!\!\int_{\mathbb{R}^N}  u_t \,
u_\tau^{-d} \psi^{k} \, dx \, dt
 + k  \int_0^\infty\!\!\int_{\mathbb{R}^N} \psi^{k-1}u_\tau^{-d} |\mathcal{A}(x,u,\nabla u)| |\nabla\psi| \, dx \, dt.
\end{aligned}\end{equation}
Now consider the first term on the right-hand side  of \eqref{ineq0} and we claim that
\begin{equation} \label{part1}
\int_0^\infty\!\!\int_{\mathbb{R}^N}  u_t \, u_{\tau}^{-d} \, \psi^k \, dx \, dt \leq  \frac{k}{1 - d} \int_0^\infty\!\!\int_{\mathbb{R}^N} \, u_{\tau}^{1 - d}
 \psi^{k - 1}  |\psi_t| \, dx \, dt.
\end{equation}
Indeed, since $(u_{\tau})_t = u_t$, by definition of $u_\tau$, we have
$$
\begin{aligned}
\int_0^\infty\int_{\mathbb{R}^N} u_t \, u_{\tau}^{-d} \psi^k \, dx \, dt &=\int_0^\infty\int_{\mathbb{R}^N} (u_{\tau})_t \; u_{\tau}^{-d} \; \psi^{k} \, dx \, dt = \int_{S}
\frac 1{1- d}\bigl( {u_{\tau}^{1 - d}} \bigr)_t \psi^{k} \, dx \, dt \\
&= \frac{1}{1 - d}  \int_0^\infty\int_{\mathbb{R}^N} \bigg[ (u_{\tau}^{1 - d} \; \psi^k)_t - u_{\tau}^{1 - d}
(\psi^k)_t \bigg] \, dx \, dt  \\
&= - \, \frac{1}{1 - d}\int_{\mathbb R^N} \, u_{\tau}^{1 - d}(x,0) \psi^k(x,0) \, dx -
\frac{k}{1 - d} \int_0^\infty\int_{\mathbb{R}^N} \, u_{\tau}^{1 - d} \; \psi^{k - 1}  \psi_t \, dx \, dt \\
\end{aligned}
$$
where the last equality is due to $\psi \in C^{1}_{c}(\mathbb R^N\times \mathbb [0,\infty))$, hence $ \lim_{t \to \infty} \psi(x,t) = 0.$ Consequently, $\eqref{part1}$ follows immediately from
$u_{\tau}(x,0) =  u_0(x) + \tau > 0$, $\psi\ge0$ and $0<d < 1$.

Further, by Young inequality (see also \cite[Lemma 4.1]{FL}) we have 
$$
u_\tau^{-d} \psi^{k-1} |\mathcal A(x,u,\nabla u)|| \nabla \psi| \leq 
\frac{dc_0}{2 k} u_\tau^{-d-1}\psi^{k}
|\mathcal A(x,u,\nabla u)|^{m'}  + Cu_\tau^{ m-d - 1}\psi^{k - m} |\nabla \psi|^m.
$$
Thanks to property (i) in the definition of $\mathcal S$ and the fact that $u_\tau\ge\tau>0$ and $0\le\psi\le1$, we have
$$ u_\tau^{-d-1}\ \psi^{k}|\mathcal{A}(x,u,\nabla u)|^{m'} 
\in L^1_{loc}(\mathbb{R}^N\times [0,\infty)).$$
Thus, a combination of   \eqref{ineq0} and \eqref{part1} yields
$$
\begin{aligned} 
\int_0^\infty\!\!\int_{\mathbb{R}^N} (K\ast u^p)u^q & u_\tau^{-d} \psi^{k} \, dx \, dt +
\int_0^\infty\!\!\int_{\mathbb R^N}  u_{\tau}^{- d - 1} \psi^{k} |\mathcal{A}(x,u,\nabla u)|^{m'}\, dx \,  dt  \\
&\leq   c_1\int_0^\infty\!\!\int_{\mathbb R^N}  u_{\tau}^{1 - d}\psi^{k - 1}|\psi_t| \, dx  dt + c_2 \int_0^\infty\!\!\int_{\mathbb R^N}  u_\tau^{ m- d- 1}\psi^{k - m} |\nabla \psi|^m  dx \, dt,
\end{aligned}$$
with $c_1,c_2>0$. Since all the exponents of $u$ on the right hand side are positive and \eqref{Summ_with_d}, we let  $\tau\to0$ and apply Fatou and Lebesgue theorems, to obtain \eqref{ineq2.0}. 
\end{proof}

\begin{lemma}\label{lp2}

Let $\ell>1$ and $u\geq 0$ be a solution of \eqref{main} such that $u^\ell\in L^1_{loc}(\mathbb{R}^N\times [0, \infty)).$ 

Define 
\begin{equation}\label{Jdef}
J(t)=\int_{\mathbb R^N}u^\ell (x,t)\psi^k(x, t)dx,
\end{equation}
where $\psi$ is given by \eqref{testfunction}. Then,
\begin{equation}\label{ineqE}\begin{aligned}
\int_0^\infty\!\!\int_{\mathbb R^N} (K\ast u^p)u^q \psi^k \, dx \, dt \le & \; c_1\biggl(\int_0^{2R^\gamma} J(t)\biggr)^{1/\ell} \cdot
R^{\frac{N+\gamma}{\ell'}-\gamma}\\
&\;+c_2\biggl(\int_0^{2R^\gamma} J(t)\biggr)^{\frac{2(m-1)}{m\ell}}\cdot
R^{(N+\gamma)\bigl(1-\frac{2}{m'\ell}\bigr) -1-\frac{\gamma}{m'}}\\
&\;+c_3 \biggl(\int_0^{2R^\gamma} J(t)\biggr)^{\frac{m-1}{\ell}}
\cdot R^{(N+\gamma)\bigl(1-\frac{m-1}{\ell}\bigr)-m},
\end{aligned}
\end{equation}
for some constants $c_1,c_2,c_3>0$.
\end{lemma}

\begin{proof} We first choose $\varphi=\psi^k$ in the weak formulation \eqref{weak_form1}, with  $\psi$ given by \eqref{testfunction}. Since $u_0, \varphi\ge0$,  we find
\begin{equation}\label{weak_k}
 \int_0^\infty\!\!\int_{\mathbb R^N} (K\ast u^p)u^q \psi^k \, dx \, dt \leq  - \!\int_0^\infty\!\!\int_{\mathbb R^N} u \, (\psi^k)_{t} \, dx \, dt 
 +\!\!  \iint\limits_{{\rm supp}\,(\nabla \psi)}  |\mathcal{A}(x,u,\nabla u)||\nabla\psi^k| dx \, dt.
\end{equation}
From now on, the constant $c_1,\, c_2$ and $c_3$ will assume  different values.

By   H\"older inequality  we obtain
\begin{equation}\label{weak_kk}
\begin{aligned}
\int_0^\infty\!\!\int_{\mathbb R^N} (K\ast u^p)u^q \psi^k \, dx \, dt \leq &  
\, C \biggl( \iint\limits_{{\rm supp}\,  (\psi_t)}  u^\ell\psi^kdx dt\biggr)^{1/\ell}\cdot \biggl(\iint\limits_{{\rm supp}\,  (\psi_t) }  
\psi^{k-\ell'}|\psi_t|^{\ell'}\biggr)^{1/\ell'} \\
& +\!\! \iint\limits_{{\rm supp}(\nabla \psi)} |\mathcal{A}(x,u,\nabla u)||\nabla\psi^k| dx \, dt.
\end{aligned}
\end{equation}
We next estimate the last integral in \eqref{weak_kk}. First, by H\"older's inequality we have
$$
\begin{aligned}
\iint\limits_{{\rm supp}(\nabla \psi)} |\mathcal{A}(x,u,\nabla u)||\nabla\psi^k| dx \, dt\leq & 
k \biggl(\int_0^\infty\!\!\int_{\mathbb R^N}u^{-d-1}\psi^k|\mathcal{A}(x,u,\nabla u)|^{m'} dx\, dt \biggr)^{1/m'}\\
& \cdot
\biggl(\int_0^\infty\!\!\int_{\mathbb R^N}u^{(d+1)(m-1)}\psi^{k-m}|\nabla \psi |^m dx\, dt\biggr)^{1/m}.
\end{aligned}
$$
Next, we use the estimate \eqref{ineq2.0} from Lemma \ref{lp1} and the standard inequality $(a+b)^r\le c(a^r+b^r)$ for $a,b,r>0$. We deduce

\begin{equation}\label{ineq0_Hold}
\begin{aligned}\qquad\quad\int_0^\infty\!\!\int_{\mathbb R^N} |\mathcal{A}(x,u,\nabla u)|&|\nabla\psi^k| dx \, dt\le 
c_1 \biggl(\int_0^\infty\!\!\int_{\mathbb R^N}  u^{1 - d}\psi^{k - 1}|\psi_t| \, dx  dt \biggr)^{1/m'}\\
&\,\,\cdot
\biggl(\int_0^\infty\!\!\int_{\mathbb R^N}u^{(d+1)(m-1)}\psi^{k-m}|\nabla \psi |^m dx\, dt\biggr)^{1/m}
\\&+ c_2\biggl(\int_0^\infty\!\!\int_{\mathbb R^N}  u^{ m- d- 1}\psi^{k - m} |\nabla \psi|^m  dx \, dt\biggr)^{1/m'}\\
&\,\,\cdot
\biggl(\int_0^\infty\!\!\int_{\mathbb R^N}u^{(d+1)(m-1)}\psi^{k-m}|\nabla \psi |^m dx\, dt\biggr)^{1/m}.
\end{aligned}\end{equation}
We now use H\"older inequality in all the factors of the right hand side, so that
\begin{equation}\label{ineqA}
\begin{aligned}
\int_0^\infty\!\!\int_{\mathbb R^N}  u^{1 - d}\psi^{k - 1}|\psi_t| \, dx  dt
&\leq \Big(\iint\limits_{{\rm supp}\,  (\psi_t)}  u^\ell \psi^k\Big)^{1/\sigma}
\Big(\iint\limits_{{\rm supp}\,  (\psi_t)} \psi^{k-\sigma'}|\psi_t|^{\sigma'}\Big)^{1/\sigma'},\\
\end{aligned}
\end{equation}
where
$$\sigma=\frac{\ell}{1-d},\qquad \sigma'=\frac \ell{\ell+d-1};$$
\begin{equation}\label{ineqB}
\begin{aligned}
\int_0^\infty\!\!\int_{\mathbb R^N}u^{(d+1)(m-1)}&\psi^{k-m}|\nabla \psi |^m dx\, dt\\
&\leq \Big( \iint\limits_{{\rm supp}\,  (\nabla \psi)}    u^\ell \psi^k\Big)^{1/\eta}
\Big(\int_0^\infty\!\!\int_{\mathbb R^N}   \psi^{k-m\eta'}|\nabla \psi|^{m\eta'}\Big)^{1/\eta'},
\end{aligned}
\end{equation}
 where
\begin{equation}\label{eta}
\eta=\frac{\ell}{(d+1)(m-1)},\quad \eta'=\frac{\ell}{\ell-m+1-d(m-1)};
\end{equation}
\begin{equation}\label{ineqC}
\begin{aligned}
\int_0^\infty\!\!\int_{\mathbb R^N}  u^{m-1-d} \psi^{k-m}|\nabla \psi |^m
&\leq \Big(\iint\limits_{{\rm supp}\,  (\nabla \psi)}  u^\ell \psi^k\Big)^{1/\theta}
\Big(\int_0^\infty\!\!\int_{\mathbb R^N}   \psi^{k-m\theta'}|\nabla \psi|^{m\theta'}\Big)^{1/\theta'},\\
\end{aligned}
\end{equation}
where
\begin{equation}\label{theta}
\theta=\frac{\ell}{m-d-1},\quad \theta'=\frac{\ell}{\ell-m+d+1}.
\end{equation}
We next replace \eqref{ineqA}, \eqref{ineqB} and \eqref{ineqC} in \eqref{ineq0_Hold}. With $J(t)$ defined in \eqref{Jdef} we obtain

$$
\begin{aligned}&\int_0^\infty\!\!\int_{\mathbb R^N} |\mathcal{A}(x,u,\nabla u)||\nabla\psi^k| dx \, dt\le\\
&\;\; c_1 \biggl(\int_0^\infty J(t)\biggr)^{\frac1{m'\sigma}+\frac1{m\eta}}
\Big(\int_0^\infty\!\!\int_{\mathbb R^N}  \psi^{k-\sigma'}|\psi_t|^{\sigma'}\Big)^{1/m'\sigma'}
\Big(\int_{\mathbb R^N}  \psi^{k-m\eta'}|\nabla \psi|^{m\eta'}\Big)^{1/m\eta'}\\
&+c_2 \biggl(\int_0^\infty J(t)\biggr)^{\frac1{m'\theta}+\frac1{m\eta}}
\Big(\int_0^\infty\!\!\int_{\mathbb R^N}  \psi^{k-m\theta'}|\nabla \psi|^{m\theta'}\Big)^{1/m'\theta'}
\Big(\int_0^\infty\!\!\int_{\mathbb R^N}  \psi^{k-m\eta'}|\nabla \psi|^{m\eta'}\Big)^{1/m\eta'}.
\end{aligned}
$$
Inserting the above inequality in \eqref{weak_kk}  and using \eqref{PSI} and \eqref{PSII} together with the fact that $J(t)=0$ for $t\geq 2R^\gamma$ {from the definition of the test function $\psi$},   we find
\begin{equation}\label{ineqD}
\begin{aligned}
\int_0^\infty\!\!\int_{\mathbb R^N} (K\ast u^p)u^q \psi^k \, dx \, dt \le & \; c_1\biggl(\int_0^{2R^\gamma} J(t)\biggr)^{1/\ell} \cdot
R^{\frac{N+\gamma-\gamma\ell'}{\ell'}}\\
&\;+c_2\biggl(\int_0^{2R^\gamma} J(t)\biggr)^{\frac1{m'\sigma}+\frac1{m\eta}}\cdot
R^{\frac{N+\gamma-\gamma\sigma'}{m'\sigma'}+\frac{N+\gamma-m\eta'}{m\eta'}}\\
&\; +c_3 \biggl(\int_0^{2R^\gamma} J(t)\biggr)^{\frac1{m'\theta}+\frac1{m\eta}}
\cdot R^{\frac{N+\gamma-m\theta'}{m'\theta'}+\frac{N+\gamma-m\eta'}{m\eta'}}\\
= & \;c_1\biggl(\int_0^{2R^\gamma} J(t)\biggr)^{1/\ell} \cdot
R^{\frac{N+\gamma-\gamma\ell'}{\ell'}}\\
&\; +c_2\biggl(\int_0^{2R^\gamma} J(t)\biggr)^{\frac1{m'\sigma}+\frac1{m\eta}}\cdot
R^{(N+\gamma)\bigl (\frac 1{m'\sigma'}+\frac1{m\eta'}\bigr) -1-\frac{\gamma}{m'}}\\
&\;+c_3 \biggl(\int_0^{2R^\gamma} J(t)\biggr)^{\frac1{m'\theta}+\frac1{m\eta}}
\cdot R^{(N+\gamma)\bigl(\frac1{m'\theta'}+\frac1{m\eta'}\bigr)-m}.
\end{aligned}
\end{equation}
Consequently, being
$$\frac 1{m'\sigma }+\frac 1{m\eta }=\frac{2(m-1)}{m\ell},  \qquad \frac 1{m'\theta } +\frac 1{m\eta }=\frac{m-1}{\ell}, $$
$$\frac1{m'\sigma'}+\frac1{m \eta'}=\frac{m\ell-2(m-1)}{m\ell},\qquad  \frac1{m'\theta'}+\frac1{m \eta'}=\frac{\ell-m+1}{\ell}, $$
 using \eqref{ineqD} and the values of the parameters involved, we get the estimate \eqref{ineqE}.
\end{proof}

\begin{lemma}\label{lp3}
Let $u\geq 0$ be a solution of \eqref{main} and $J$ be given by \eqref{Jdef} with $\ell=(p+q)/2>1$. Then
\begin{equation}\label{ineq_final}
\int_0^{2R^\gamma} J(t)^2 dt \le c\bigl(R^{\alpha_1}+R^{\alpha_2}+R^{\alpha_3}\bigr),
\end{equation}
where $c>0$ is a constant and
$$
\begin{aligned}
\alpha_1& =\frac{(N+\beta)(p+q)-2N-\gamma}{p+q-1}\,,\\
\alpha_2&=\frac{(N+\beta-1)m(p+q)-4N(m-1)+\gamma[p+q-2(m-1)]}{m(p+q)-2(m-1)}\,,\\
\alpha_3& =\frac{(N+\beta-m)(p+q)-2N(m-1)+\gamma(p+q-m+1)}{p+q-m+1}.
\end{aligned}
$$
\end{lemma}

\begin{proof} 
First note that by \eqref{uLp+q_2}, we can choose $\ell=(p+q)/2$ so that the requirement of Lemma \ref{lp2} is satisfied.
Now, observe that supp$\, (\psi^k)\subset B_{2R}(0)\times [0, 2R^\gamma)$.
If $(x,t)\in B_{2R}(0)\times [0, \infty)$, then, in the same way as we estimated \eqref{KKK1} we find 
\begin{equation}\label{KKK2} 
(K\ast u^p)(x,t)
\geq CR^{-\beta}\int_{B_{2R}(0) } u^p(y,t)dy
\geq CR^{-\beta}\int_{\mathbb R^N } u^p(y,t)\psi^k(y,t)dy,
\end{equation}
provided $R>\rho$ is large enough.

Furthermore, for $\ell=(p+q)/2>1$, by H\"older's inequality we have
$$
\begin{aligned}
\biggl(\iint_{\mathbb R^N\times \mathbb R^N}u^p(y,t)\psi^k(y,t) & u^{q}(x,t) \psi^{k}(x,t)  \, dx \, dy\biggr) ^2\\
=\; &\biggl(\iint_{\mathbb R^N\times \mathbb R^N}u^p(y,t)\psi^k(y,t) u^{q}(x,t) \psi^{k}(x,t) \, dx \, dy\biggr)\\
\; &\cdot \biggl(\iint_{\mathbb R^N\times \mathbb R^N}u^p(x,t)\psi^k(x,t) u^{q}(y,t) \psi^{k}(y,t) \, dx \, dy\biggr) \\
\geq \; &
\biggl(\iint_{\mathbb R^N\times \mathbb R^N}u^{\frac{p+q}{2}} (x,t)u^{\frac{p+q}{2}} (y,t)\psi^k(x,t)  \psi^{k}(y,t) \, dx \, dy\biggr)^2\\
=\; & \biggl(\int_{\mathbb R^N}u^{\ell} (x,t)\psi^k(x,t)  \, dx \biggr)^4=J(t)^4,
\end{aligned}
$$
where $J$ is given by \eqref{Jdef} with $\ell=(p+q)/2>1$. Hence,
$$
\iint_{\mathbb R^N\times \mathbb R^N}u^p(y,t)\psi^k(y,t)  u^{q}(x,t) \psi^{k}(x,t)  \, dx \, dy\geq J^2(t).
$$
Hence, using the above estimate and \eqref{KKK2} we deduce
\begin{equation}\label{JJ1}\begin{aligned}
& \int_{\mathbb R^N} (K\ast u^p)u^q \psi^k \, dx \, dt\\
&\qquad\quad\ge C R^{-\beta}\iint_{\mathbb R^N\times \mathbb R^N}u^p(y,t)\psi^k(y,t)  u^{q}(x,t) \psi^{k}(x,t)  \, dx \, dy\\
&\geq  C R^{-\beta} J(t)^2.
\end{aligned}\end{equation}

As observed before inequality \eqref{ineqD}, we have $J(t)=0$ if $t\geq 2R^{\gamma}$.
Thus, \eqref{JJ1} and inequality \eqref{ineqE} in Lemma \ref{lp2} yields
\begin{equation}\label{JJ2}\begin{aligned}
\int_0^{2R^\gamma} J(t)^2 dt &\le c_1\biggl(\int_0^{2R^\gamma} J(t) dt\biggr)^{1/\ell} \cdot
R^{\frac{N+\gamma}{\ell'}-\gamma+\beta}\\
&\quad+c_2\biggl(\int_0^{2R^\gamma} J(t) dt\biggr)^{\frac{2(m-1)}{m\ell}}\cdot
R^{(N+\gamma)\bigl(1-\frac{2}{m'\ell}\bigr) -1-\frac{\gamma}{m'}+\beta}\\
&\quad+c_3 \biggl(\int_0^{2R^\gamma} J(t) dt\biggr)^{\frac{m-1}{\ell}}
\cdot R^{(N+\gamma)\bigl(1-\frac{m-1}{\ell}\bigr)-m+\beta}\\
\end{aligned}
\end{equation} 
In addition, by H\"older's inequality we find
\begin{equation}\label{Holder}
\int_0^{2R^\gamma} J(t) dt\leq \sqrt{2}R^{\gamma/2}\Big(\int_0^{2R^\gamma} J(t)^2 dt\Big)^{1/2},
\end{equation}
so that inequality \eqref{JJ2} implies
 \begin{equation}\label{JJ2_1}\begin{aligned}
\int_0^{2R^\gamma} J(t)^2 dt &\le c_1\biggl(\int_0^{2R^\gamma} J(t)^2 dt\biggr)^{1/2\ell} \cdot
R^{\frac{N+\gamma}{\ell'}-\gamma+\beta+\frac{\gamma}{2\ell}}\\
&\quad+c_2\biggl(\int_0^{2R^\gamma} J(t)^2 dt\biggr)^{\frac{m-1}{m\ell}}\cdot
R^{(N+\gamma)\bigl(1-\frac{2}{m'\ell}\bigr) -1-\frac{\gamma}{m'}+\beta+\frac{\gamma}{m'\ell}}\\
&\quad+c_3 \biggl(\int_0^{2R^\gamma} J(t)^2 dt\biggr)^{\frac{m-1}{2\ell}}
\cdot R^{(N+\gamma)\bigl(1-\frac{m-1}{\ell}\bigr)-m+\beta+\frac{\gamma(m-1)}{2\ell}}.
\end{aligned}
\end{equation}
Now, applying Young inequality we arrive to
\begin{equation}\label{JJ2_2}\begin{aligned}
\int_0^{2R^\gamma} J(t)^2 dt &\le \frac 16\int_0^{2R^\gamma} J(t)^2 dt +c_4
R^{\bigl[\frac{N+\gamma}{\ell'}+\beta-\frac\gamma{(2\ell)'}\bigr](2\ell)'}\\
&\quad+\frac 16\int_0^{2R^\gamma} J(t)^2 dt +c_5
R^{\bigl[(N+\gamma)\bigl(1-\frac{2}{m'\ell}\bigr) -1-\frac{\gamma}{m'\ell'}+\beta\bigr]\bigl(\frac{m\ell}{m-1}\bigr)'}\\
&\quad+\frac 16 \int_0^{2R^\gamma} J(t)^2 dt
+c_6 R^{\bigl[(N+\gamma)\bigl(1-\frac{m-1}{\ell}\bigr)-m+\beta+\frac{\gamma(m-1)}{2\ell}\bigr]\bigl(\frac{2\ell}{m-1}\bigr)'},
\end{aligned}
\end{equation}
which yields \eqref{ineq_final},   being
$$\ell=\frac{p+q}2,\quad \biggl(\frac{m\ell}{m-1}\biggr)'=\frac{m(p+q)}{m(p+q)-2(m-1)}, \quad 
\biggl(\frac{2\ell}{m-1}\biggr)'=\frac{p+q}{p+q-m+1}.
$$
\end{proof}

{\it Proof of Theorem \ref{thmain}.} The proof  will be carried out by taking a specific value of 
$\gamma\geq 1$ in the definition of the cut off functions in \eqref{cutoff}. To this aim, 
we write $\alpha_1$, $\alpha_2$ and $\alpha_3$ so that
$$\alpha_1=\frac{\Upsilon}{p+q-1},\quad \alpha_2=\frac{2(m-1)}{m(p+q)-2(m-1)}\Upsilon, \quad \alpha_3=\frac{m-1}{p+q-m+1}\Upsilon$$
for a certain real constant $\Upsilon$ to be determined.
Since
$$
\begin{aligned}
\alpha_2&=\frac{2(m-1)}{m(p+q)-2(m-1)}\biggl[\frac{N+\beta-1}{2(m-1)}m(p+q)-2N+\gamma\,\frac{p+q-2(m-1)}{2(m-1)}\biggr]\,,\\
\alpha_3 & =\frac{m-1}{p+q-m+1}\biggl[\frac{N+\beta-m}{m-1}(p+q)-2N+\gamma\,\frac{p+q-m+1}{m-1}\biggr]\,,
\end{aligned}
$$
we obtain that
\begin{equation}\label{gamma_iupsilon}\gamma=(m-2)(N+\beta)+m, \quad \Upsilon=(N+\beta)(p+q)-[Nm+\beta(m-2)+m].
\end{equation}
Consequently, for
$$p+q<m-1+\frac{N-\beta+m}{N+\beta}
$$
we obtain $\Upsilon<0$, that is, $\alpha_1,\,\alpha_2, \, \alpha_3<0$. By letting $R\to\infty$ in  \eqref{ineq_final}, it follows that
\begin{equation}\label{esst}
\int_0^{ \infty} J(t)^2 dt=0.
\end{equation}
In turn, using the definition of $J$ in \eqref{Jdef} and the fact that $\psi\equiv1$ in $B_R\times[0, R^\gamma)$, 
this easily yields $u\equiv 0$, a contradiction. 

If
\begin{equation}\label{critical} 
p+q=m-1+\frac{N-\beta+m}{N+\beta},
\end{equation}
then inequality \eqref{ineq_final} for $R\to\infty$ implies $$J\in L^2(0, \infty).$$

Note that 
$$
{\rm supp}\; (\nabla \psi)=\bigl(B_{2R}\setminus B_R\bigr)\times (0, 2R^\gamma)\quad \mbox{ and }\quad  
{\rm supp}\, (\psi_t)=B_{2R}\times (R^\gamma,2R^\gamma).
$$
Define next
$$
\Theta_R=\max\left\{ \int_{R^\gamma}^{2R^\gamma}\biggl(\int_{B_{2R}}u^\ell\psi^kdx\biggr)^2dt\,,
\int_0^{2R^\gamma} \biggl(\int_{B_{2R}\setminus B_R} u^\ell\psi^kdx\biggr)^2dt
\right\}.
$$
Since $J\in L^2(0, \infty)$, it follows that $\Theta_R\to 0$ as $R\to \infty$.
We now retake the estimate \eqref{weak_kk} and all the calculations that follow up to 
\eqref{JJ2_1} to derive
$$
\int_0^{2R^\gamma} J(t)^2 dt \le C\biggl \{\Theta_R^{\frac{1}{2\ell}} +\Theta_R^{\frac{m-1}{m\ell}} + \Theta_R^{\frac{m-1}{\ell}} \biggr\},
$$  
being zero all the exponents of $R$ on the right hand side of \eqref{JJ2_1}, by \eqref{critical} 
and \eqref{gamma_iupsilon} where $\Upsilon=0$.
Now, letting $R\to \infty$ in the above estimate it follows that $J\equiv 0$ and then $u\equiv 0$ which again contradicts our assumption. 
\qed

\noindent{\bf Proof of Corollary \ref{cornu}.} In Lemma  \ref{lp3} we choose the value of $\gamma$ as follows
$$\gamma=\frac{(N-\beta)(m-2)+m(p+q-1)}{p+q-m+1}.$$ 
Then
$$\alpha_1=\alpha_2=\alpha_3:=\alpha=N+\beta-\frac{N-\beta+m}{p+q-m+1}.$$
Starting with \eqref{weak_form1} and going through the same estimates as above without removing the term $\int_{\mathbb R^N}u_0(x)\varphi(x,0)dx$, the inequality
\eqref{ineq_final} changes to
$$\int_0^{2R^\gamma} J(t)^2 dt \le c_1R^{\alpha}-c_2\int_{B_R}u_0(x)dx,$$
where we have used $\psi\equiv1$ in $B_R$.
Thanks to \eqref{u_0_nu} we arrive to 
$$\int_0^{2R^\gamma} J(t)^2 dt \le c\bigl(R^{\alpha}-R^\nu\bigr).$$
Since condition \eqref{critic_nu} forces $\alpha<\nu$, the above estimate yields a contradiction as $R\to\infty$ in both cases $\nu>0$ and $\nu=0$.\qed

\section{Proof of Theorem \ref{thmain2}}

Suppose by contradiction that \eqref{main2} admits a nonnegative nontrivial solution $u\in \mathcal{S}$ such that 
\begin{equation}\label{Kextra}
(K\ast u^p)u^{q+d} \in  L^{1}_{loc}(\mathbb R^N\times [0, \infty)).
\end{equation}
In the same way as in Lemma \ref{lp0} (we only have to replace the exponent $q$ by $q+d$) we deduce
$$
u^{(p+q+d)/2}\in L^1_{loc}(\mathbb{R}^N\times [0,\infty)).
$$
From \eqref{eqd} we have
$$
\frac{p+q+d}{2}>\max\{d+1\, , \, d+m-1\}.
$$
Thus, by H\"older's inequality we deduce
$$
u^{d+m-1}\,,\, u^{d+1}\in L^1_{loc}(\mathbb{R}^N\times [0,\infty)).
$$
This will ensure that all integrals in this section are finite.

We start with the following result which is a counterpart of Lemma \ref{lp1}.

\begin{lemma}\label{lpp1}
 Let $u\in \mathcal{S}$ be a nonnegative solution of \eqref{main2} satisfying \eqref{Kextra} and let $\psi$ be defined by \eqref{testfunction}. Then,
\begin{equation}\label{ineqq2.0}
\begin{aligned} 
\int_0^\infty\int_{\mathbb{R}^N} & (K\ast u^p)u^{q+d} \psi^{k} \, dx \, dt+
\int_0^\infty\!\!\int_{\mathbb R^N}  u^{d - 1} \psi^{k} |\mathcal{A}(x,u,\nabla u)|^{m'}\, dx \,  dt  \\
&\leq   c_1\int_0^\infty\!\!\int_{\mathbb R^N}  u^{d+1}\psi^{k - 1}|\psi_t| \, dx  dt + 
c_2 \int_0^\infty\!\!\int_{\mathbb R^N}  u^{ d+m- 1}\psi^{k - m} |\nabla \psi|^m  dx \, dt,
\end{aligned}\end{equation}
for some constants $c_1,c_2>0$ and $d>0$ satisfies \eqref{eqd}.  
\end{lemma}
The proof follows line by line  that of Lemma \ref{lp1} in which we replace $d$ by $-d$. 

Similar to the estimate \eqref{JJ1} we find
\begin{equation}\label{L1}\begin{aligned}
\int_{\mathbb R^N} (K\ast u^p)u^{q+d} \psi^k \, dx \, dt
\ge R^{-\beta} L(t)^2,
\end{aligned}\end{equation}
where
\begin{equation}\label{Ldef}
L(t)=\int_{\mathbb R^N}u^\tau (x,t)\psi^k(x, t)dx,
\end{equation}
where $\tau=(p+q+d)/2>1$ and $\psi$ is defined by \eqref{testfunction}. In the same way as we deduced \eqref{ineqA}, by H\"older's inequality we find

\begin{equation}\label{ineq2A}
\begin{aligned}
\int_0^\infty\!\!\int_{\mathbb R^N}  u^{d+1}\psi^{k - 1}|\psi_t| \, dx  dt
&\leq \Big(\int_0^\infty\!\!\int_{\mathbb R^N}   u^\tau \psi^k\Big)^{1/\sigma}
\Big(\int_0^\infty\!\!\int_{\mathbb R^N} \psi^{k-\sigma'}|\psi_t|^{\sigma'}\Big)^{1/\sigma'}\\[0.05in]
&= \Big(\int_0^{2R^\gamma}L(t) dt\Big)^{1/\sigma}
\Big(\int_0^\infty\!\!\int_{\mathbb R^N} \psi^{k-\sigma'}|\psi_t|^{\sigma'}\Big)^{1/\sigma'}\\[0.05in]
&\leq c R^{\gamma/(2\sigma)} \Big(\int_0^{2R^\gamma}L(t)^2 dt\Big)^{1/(2\sigma)}
\Big(\int_0^\infty\!\!\int_{\mathbb R^N} \psi^{k-\sigma'}|\psi_t|^{\sigma'}\Big)^{1/\sigma'}\\[0.05in]
&\leq C \Big(\int_0^{2R^\gamma}L(t)^2 dt\Big)^{1/(2\sigma)}R^{\frac{2N+\gamma}{2\sigma'}-\frac{\gamma}{2}},\\
\end{aligned}
\end{equation}
where $c>0$ is a constant and 
$$\sigma=\frac{\tau}{d+1},\qquad \sigma'=\frac \tau{\tau-d-1}.$$
Note that since $\Gamma(d)>2$ we have $\sigma>1$. 

By H\"older's inequality, similarly  to the estimate \eqref{ineqC}, we have
\begin{equation}\label{ineq2C}
\begin{aligned}
\int_0^\infty\!\!\int_{\mathbb R^N}  u^{d+m-1} & \psi^{k-m}|\nabla \psi |^m
\leq \Big(\int_0^\infty\!\!\int_{\mathbb R^N}   u^\tau \psi^k\Big)^{1/\theta}
\Big(\int_0^\infty\!\!\int_{\mathbb R^N}   \psi^{k-m\theta'}|\nabla \psi|^{m\theta'}\Big)^{1/\theta'}\\[0.05in]
&= \Big(\int_0^{2R^\gamma}L(t) dt\Big)^{1/\theta}
\Big(\int_0^\infty\!\!\int_{\mathbb R^N}   \psi^{k-m\theta'}|\nabla \psi|^{m\theta'}\Big)^{1/\theta'}\\[0.05in]
&\leq cR^{\gamma/(2\theta)} \Big(\int_0^{2R^\gamma}L(t)^2 dt\Big)^{1/(2\theta)}
\Big(\int_0^\infty\!\!\int_{\mathbb R^N}   \psi^{k-m\theta'}|\nabla \psi|^{m\theta'}\Big)^{1/\theta'}\\[0.05in]
&\leq C \Big(\int_0^{2R^\gamma}L(t)^2 dt\Big)^{1/(2\theta)}
R^{\frac{2N+\gamma}{2\theta'}-m+\frac{\gamma}{2}},
\\
\end{aligned}
\end{equation}
 where
\begin{equation}\label{eta2}
\theta=\frac{\tau}{d+m-1},\quad \theta'=\frac{\tau}{\tau-d-m+1},
\end{equation}
Note that from $\Gamma(d)>2$ we derive $\eta>1$. 
Take now $\gamma=m$.

Next, we use \eqref{L1}, \eqref{ineq2A} and \eqref{ineq2C} in the estimate \eqref{ineqq2.0} of Lemma \ref{lpp1} to deduce
\begin{equation}\label{la2}
\int_0^{2R^\gamma}L(t)^2 dt \leq c\Big(\int_0^{2R^\gamma}L(t)^2 dt\Big)^{\frac{1}{2\sigma}}R^{\alpha_1}+c\Big(\int_0^{2R^\gamma}L(t)^2 dt\Big)^{\frac{1}{2\theta}}R^{\alpha_2},
\end{equation}
where, by $\eqref{eqd}$ and $\gamma=m$, we have
$$
\begin{aligned}
\alpha_1&=\frac{2N+m}{2\sigma'}-\frac{m-2\beta}{2}=\frac{(N+\beta)(p+q+d)-(2N+m)(d+1)}{p+q+d}<0,\\
\alpha_2&=\frac{2N+m}{2\theta'}-\frac{m-2\beta}{2}=\frac{(N+\beta)(p+q+d)-(2N+m)(d+m-1)}{p+q+d}<0.
\end{aligned}
$$

We further apply Young's inequality in the right hand-side of \eqref{la2} and obtain
$$
\int_0^{2R^\gamma}L(t)^2 dt \leq \frac{1}{4}\int_0^{2R^\gamma}L(t)^2 dt+C R^{\alpha_1 (2\sigma)'}+\frac{1}{4}\int_0^{2R^\gamma}L(t)^2 dt +C R^{\alpha_2 (2\theta)' },
$$
that is,
$$
\int_0^{2R^\gamma}L(t)^2 dt \leq 2C\Big( R^{\alpha_1 (2\sigma)'}+ R^{\alpha_2 (2\theta)' }\Big)\longrightarrow  0\quad\mbox{ as }R\to\infty.
$$
Proceeding as in the proof of Theorem \ref{thmain} this yields $L(t)\equiv0$ in $[0,\infty)$, which gives 
$u\equiv 0$, contradiction.
\qed

\section*{Acknowledgments} RF is a member of the {\em Gruppo Nazionale per
l'Analisi Ma\-te\-ma\-ti\-ca, la Probabilit\`a e le loro Applicazioni}
(GNAMPA) of the {\em Istituto Nazionale di Alta Matematica} (INdAM) 
and she was partly supported by  {\em Fondo Ricerca di
Base di Ateneo Esercizio} 2017-19 of the University of Perugia, named
{\em Problemi con non linearit\`a dipendenti dal gradiente}
and by INdAM-GNAMPA Project 2020 titled
{\em Equazioni alle derivate parziali : problemi e modelli}
(Prot\_U-UFMBAZ-2020).

\end{document}